\documentclass{amsart}

\usepackage{epsfig}
\usepackage{psfrag}

\usepackage{infix-RPN}
\usepackage{pst-plot,pst-infixplot,pstricks,graphicx}
\usepackage{amssymb,latexsym,amsthm,amsfonts,color,fancyhdr}

\newtheorem{theorem}{Theorem}[section]

\newtheorem{lemma}[theorem]{Lemma}

\newtheorem{conjecture}[theorem]{Conjecture}

\theoremstyle{remark}
\newtheorem{remark}[theorem]{Remark}

\numberwithin{equation}{section}

\makeatletter
\def\BState{\State\hskip-\ALG@thistlm}
\makeatother

\def\downbar#1{
\setbox10=\hbox{$#1$}
            \dimen10=\ht10 \advance\dimen10 by 2.5pt
            \ifdim \dimen10<15pt 
               \advance\dimen10 by -0.5pt
               \dimen11=\dimen10
               \advance\dimen10 by 2.5pt
               \lower \dimen11
            \else \lower \ht10 \fi
            \hbox {\hskip 1.5pt \vrule height \dimen10 depth \dp10}}
\def\upbar#1{
\setbox10=\hbox{$#1$}
            \dimen10=\ht10 \advance\dimen10 by \dp10 \advance\dimen10 by 2.5pt
            \ifdim \dimen10<15pt 
                \advance\dimen10 by 2pt \fi
            \raise 2.5pt \hbox {\hskip -1.5pt \vrule height \dimen10}}

\begin{document}

\title[Orthogonal polynomials related with the Askey-Wilson operator]
{A structure relation for some specific orthogonal polynomials}
\author{D. Mbouna}
\address{University of Almer\'ia, Department of Mathematics, Almer\'ia, Spain}
\email{mbouna@ual.es}



\date{\today}

\thanks{I dedicate this work to the memory of my beloved friend and mentor J. Petronilho who passed away unexpectedly last year. To him, my eternal gratitude.}


\date{\today}


\maketitle


\begin{abstract}
By characterizing all orthogonal polynomials sequences $(P_n)_{n\geq 0}$ for which
$$
(ax+b)(\triangle +2\,\mathrm{I})P_n(x(s-1/2))=(a_n x+b_n)P_n(x)+c_n P_{n-1}(x),\quad n=0,1,2,\dots,
$$
where $\,\mathrm{I}$ is the identity operator, $x$ defines a $q$-quadratic lattice, $\triangle f(s)=f(s+1)-f(s)$, and $(a_n)_{n\geq0}$, $(b_n)_{n\geq0}$ and $(c_n)_{n\geq0}$ are sequences of complex numbers, we derive some new structure relations for some specific families of orthogonal polynomials.
\end{abstract}

\section{Introduction}
Orthogonal polynomials theory is an interesting branch of mathematics. It has applications in other related fields (statistics, approximation theory, number theory, ..., etc). The approach with lattices was most welcome because this is useful to describe in an unified way families of orthogonal polynomial sequences (OPS) including classical ones. For a recent reference on the subject we refer the reader to \cite{KCDMJP2022} including  some reference therein, where some properties of the so-called Askey-Wilson operator and Askey-Wilson polynomials are studied. Despite the fact that classical OPS (on lattices) constitute the most studied class of OPS, they are still some interesting unsolved problems (see \cite[p. 653]{IsmailBook2005}). The value of this contribution is then to study some structure relations and so to obtain characterization theorems for some specific families of OPS. This will certainly give ideas on some appropriate basis to use when dealing with these operators. 

 We consider the Askey-Wilson operator defined by,
\begin{align}
(\mathcal{D}_q  f)(x)=\frac{\breve{f}\big(q^{1/2} z\big)
-\breve{f}\big(q^{-1/2} z\big)}{\breve{e}\big(q^{1/2}z\big)-\breve{e}\big(q^{-1/2} z\big)},\quad
z=e^{i\theta}, \label{0.3}
\end{align}
where $\breve{f}(z)=f\big((z+1/z)/2\big)=f(\cos \theta)$ for each polynomial $f$ and $e(x)=x$.
Here $0<q<1$ and $\theta$ is not necessarily a real number (see \cite[p.\,300]{IsmailBook2005}). The following problem \cite[Conjecture 24.7.8]{IsmailBook2005} is a conjecture posed by M. E. H. Ismail. 
\begin{conjecture}\label{IsmailConj2478}
Let $(P_n)_{n\geq0}$ be a monic OPS and $\pi$ be a polynomial of degree at most $2$
which does not depend on $n$. If $(P_n)_{n\geq0}$ satisfies
\begin{align}
\pi(x)\mathcal{D}_q P_n (x)= (a_n x+b_n)P_n (x) + c_n P_{n-1} (x)\;, \label{0.2Dq}
\end{align}
then $(P_n)_{n\geq0}$ are continuous $q$-Jacobi polynomials, Al-Salam-Chihara polynomials,
or special or limiting cases of them. The same conclusion holds if $\pi$ has degree
$s+1$ and the condition $\eqref{0.2Dq}$ is replaced by
\begin{align}\label{0.2Dq-general}
\pi(x)\mathcal{D}_qP_n(x)=\sum_{k=-r} ^{s} c_{n,k}P_{n+k}(x)\;,
\end{align}
for positive integers $r$, $s$, and a polynomial $\pi$ which does not depend on n.
\end{conjecture}
Although \eqref{0.2Dq} is a simple relation, this problem was only solved recently due to the complexity of the Askey-Wilson operator and his properties. For instance, it is proved in \cite{Al-Salam1995} (for the case $\pi(x)=1$) and in \cite{KCDMJP2021c} that the only solutions of \eqref{0.2Dq} are the some particular cases of the Al-Salam Chihara polynomials, the Chebyshev polynomials of the first kind and the continuous $q$-Jacobi polynomials. The second part of this conjecture is disproved in \cite{KCDM2022}, where the authors provide a counterexample to \eqref{0.2Dq-general}.

The motivation of this work is the following. We consider \eqref{0.2Dq} replacing the Askey-Wilson operator by the averaging operator. That is to characterize all orthogonal polynomials sequences $(P_n)_{n\geq 0}$ such that 
\begin{align}\label{equation-to-solve-general}
\pi(x)(\triangle +2\,\mathrm{I})P_n(x(s-1/2))=(a_n x+b_n)P_n(x)+c_n P_{n-1}(x)\;,
\end{align}
where $\,\mathrm{I}$ is the identity operator, $\pi$ a polynomial of a degree at most one, $x$ is a $q$-quadratic lattice given by $x(s)=(q^{-s}+q^s)/2$ and $\triangle f(s)=f(s+1)-f(s)$. This leads to characterization of some specific families of orthogonal polynomials sequences. The aim of this work is not only to find solutions of \eqref{equation-to-solve-general}, but also to obtain appropriate polynomials basis when dealing with problems related to the Askey-Wilson operator and the averaging operator. As we are going to see, Chebyshev polynomials constitute nice basis for the mentioned operators (see Remark \ref{remark-intro} below).


 Recall that the continuous monic dual $q$-Hahn polynomials, $(H_n(x;a,b|q))_{n\geq 0}$, satisfies the following three term recurrence relation (TTRR)
\begin{align*}
xH_n(x;a,b,c|q)=H_{n+1}(x;a,b,c|q)+a_nH_{n}(x;a,b,c|q)+b_nH_{n-1}(x;a,b,c|q)\;,
\end{align*}
where $a_n=(a+a^{-1}-a(1-q^n)(1-bcq^{n-1}) -(1-abq^n)(1-acq^n)/a)/2$ and $b_n=(1-abq^n)(1-acq^n)(1-bcq^n)(1-q^{n+1})/4$, while the monic Al-Salam-Chihara polynomials, $Q_n(x;c,d|q)$,
which depend on two parameters $c$ and $d$, are characterized by
$$
\begin{array}{rcl}
xQ_{n}(x;c,d|q)&=&Q_{n+1}(x;c,d|q)+\,\mbox{$\frac12$}\,(c+d)q^n\,Q_{n}(x;c,d|q) \\ [0.25em]
&&\displaystyle\, +\,\mbox{$\frac14$}\,(1-cd q^{n-1})(1-q^n)\,Q_{n-1}(x;c,d;q),
\end{array}
$$
$n=0,1,\ldots$, provided we define $Q_{-1}(x;c,d|q)=0=H_{-1}(x;a,b,c|q)$ (see e.g. \cite{IsmailBook2005}).

The structure of the paper is as follows. Section 2 contains some preliminary results and in Section 3 our main results are stated and proved. 

\section{Preliminary results}
Recall that a monic OPS $(P_n)_{n\geq 0}$ satisfies the following TTRR:
\begin{align}
x P_n(x)=P_{n+1}(x)+B_nP_n(x)+C_nP_{n-1}(x) \quad (n=0,1,2,\dots)\;,\label{TTRR}
\end{align}
with $P_{-1}(x)=0$ and $B_n\in \mathbb{C}$ and $C_{n+1} \in \mathbb{C}\setminus \left\lbrace 0\right\rbrace$ for each $n=0,1,2,\ldots$.  Hereafter, we denote $x=x(s)=(q^{s}+q^{-s})/2$ with $0<q<1$. Taking $e^{i\theta}=q^s$ in \eqref{0.3}, $\mathcal{D}_q$ reads
\begin{equation*}
\mathcal{D}_q f(x(s))= \frac{f\big(x(s+\frac{1}{2})\big)-f\big(x(s-\frac{1}{2})\big)}{x(s+\frac{1}{2})-x(s-\frac{1}{2})}.
\end{equation*}
We define an operator $\mathcal{S}_q$ by
\begin{equation*}
\mathcal{S}_q f(x(s))=\frac{f\big(x(s+\frac{1}{2})\big)+f\big(x(s-\frac{1}{2})\big)}{2}.
\end{equation*}
Let $f$ and $g$ be two polynomials. Define
\begin{align*}
\alpha= \frac{q^{1/2}+q^{-1/2}}{2},\quad \alpha_n= \frac{q^{n/2}+q^{-n/2}}{2},\quad  \gamma_n=\frac{q^{n/2}-q^{-n/2}}{q^{1/2}-q^{-1/2}} \quad (n=0,1,\dots)\;.
\end{align*} The following properties are well known \cite{KCDMJP2022, IsmailBook2005}.
\begin{align}
\mathcal{D}_q \big(fg\big)&= \big(\mathcal{D}_q f\big)\big(\mathcal{S}_q g\big)+\big(\mathcal{S}_q f\big)\big(\mathcal{D}_q g\big), \label{def-Dx-fg} \\[7pt]
\mathcal{S}_q \big( fg\big)&=\big(\mathcal{D}_q f\big) \big(\mathcal{D}_q g\big)\texttt{U}_2  +\big(\mathcal{S}_q f\big) \big(\mathcal{S}_q g\big), \label{def-Sx-fg}
\end{align}
where $\texttt{U}_2(x)=(\alpha^2 -1)(x^2-1)$.
It is proved by induction in \cite[Proposition 2.1]{KCDMJP2022} that 
\begin{align}
\mathcal{D}_q x^n =\gamma_n x^{n-1}+\frac{n\gamma_{n-2}-(n-2)\gamma_n}{4} x^{n-3}+\cdots \quad (n=0,1,\dots)\;. \label{Dx-xn}
\end{align}

%

It is very intricate how Conjecture \ref{IsmailConj2478} was made, specially the second part. This is why we find useful to start by giving the connection between our considered structure relation \eqref{equation-to-solve-general} with equation \eqref{0.2Dq-general} appearing in the second part of the conjecture.

\begin{lemma}\label{lemma-for-alternative-relation}
Let $(P_n)_{n\geq0}$ be a monic OPS satisfying the following equation
\begin{align}\label{equation-explicit-to solve-general}
(ax-c)\mathcal{S}_qP_n(x)=(a_nx+b_n)P_n(x)+c_nP_{n-1}(x)\;.
\end{align}
Then $(P_n)_{n\geq 0}$ satisfies the following other relation
\begin{align}\label{equation-alternative-with Dx}
(ax-c)\mathtt{U}_2(x) \mathcal{D}_q P_n (x)&=r_n ^{[1]} P_{n+2}(x)+r_n ^{[2]} P_{n+1}(x)+r_n ^{[3]}P_n (x)\\
&+r_n ^{[4]}P_{n-1}(x)+r_n ^{[5]}P_{n-2}(x)\;, \nonumber
\end{align}
for each $n=0,1,2,\ldots$, where
\begin{align*}
r_n ^{[1]}&= a_{n+1}-\alpha a_n,\\
 r_n ^{[2]}&= g_{n+1}-\alpha g_n+a_n(B_n-\alpha B_{n+1}),\\
r_n ^{[3]}&= s_{n+1}-\alpha s_n+g_n(1-\alpha) B_{n}+a_{n-1}C_n-\alpha a_nC_{n+1},\\
r_n ^{[4]}&= (g_{n-1}-\alpha g_n)C_{n}+s_n(B_n-\alpha B_{n-1}),\\
 r_n ^{[5]}&= C_ns_{n-1}-\alpha C_{n-1}s_n,
\end{align*}
and $g_n =b_n +a_n B_n$, $s_n =c_n +a_n C_n$.
\end{lemma}
\begin{proof}
Let $(P_n)_{n\geq0}$ be a monic OPS satisfying \eqref{equation-explicit-to solve-general}.
From the TTRR \eqref{TTRR} fulfilled by the monic OPS $(P_n)_{n\geq 0}$ satisfying \eqref{equation-explicit-to solve-general}, we apply the operator $\mathcal{S}_q$ using \eqref{def-Sx-fg} to obtain the following equation
$$\mathtt{U}_2(x)\mathcal{D}_qP_n(x)=-\alpha x\mathcal{S}_qP_n(x)+\mathcal{S}_qP_{n+1}(x)+B_n\mathcal{S}_qP_n(x)+C_n\mathcal{S}_qP_{n-1}(x)\;.$$
Hence \eqref{equation-alternative-with Dx} holds by multiplying the above equation by the polynomial $ax-c$ and using \eqref{equation-explicit-to solve-general} together with the TTRR \eqref{TTRR}. 
\end{proof}

Next, we show that the coefficients of the associated TTRR to the monic OPS, $(P_n)_{n\geq 0}$, satisfying \eqref{equation-explicit-to solve-general}  fulfill a system of non linear equations that will be solved in the next section.

\small{
\begin{lemma}\label{system-in-genral-form-}
Let $(P_n)_{n\geq0}$ be a monic OPS satisfying \eqref{equation-explicit-to solve-general}. Then the coefficients $B_n$ and $C_{n}$ of the TTRR \eqref{TTRR} satisfied by $(P_n)_{n\geq0}$ fulfill the following system of difference equations: 
\begin{align}
&\label{eq1S-eq2S} a_{n+2}-2\alpha a_{n+1}+a_n=0\;,\; t_{n+2} -2\alpha t_{n+1}+t_{n} =0\;,\\
&\nonumber \quad\quad\quad\quad\quad t_n =\frac{c_n}{C_n} =k_1 q^{n/2}+k_2q^{-n/2}\;,  \\
&r_{n+3}B_{n+2} -(r_{n+2} +r_{n+1})B_{n+1}  +r_n B_n =0,\quad r_n= t_n +a_n -a_{n-1}\;,  \label{eq3S} \\
&\label{eq4S} r_n \left( B_{n} ^2 -2\alpha B_{n}B_{n-1} +B_{n-1}^2   \right)\\
&\nonumber \quad \quad\quad\quad\quad=(r_{n+1}+r_{n+2})(C_{n+1}-1/4) -2(1+\alpha)r_n(C_n -1/4)\\
 &\nonumber \quad\quad\quad\quad\quad+(r_{n-1}+r_{n-2})(C_{n-1} -1/4)\;, \\
&\label{eq5S} (1 -\alpha^2)b_n=2(1-\alpha)(a_nB_n+b_n)B_n ^2+(t_{n+1}\\
&\nonumber  \quad\quad\quad\quad\quad+a_{n+1}-a_{n+2})B_{n+1}C_{n+1}+(t_n+a_{n-1}-a_{n-2})B_{n-1}C_n \\
&\nonumber  \quad\quad\quad\quad\quad+\Big[(2a_n-a_{n+2}-a_{n-1})C_{n+1}+(2a_n -a_{n+1}-a_{n-2})C_n\\
&\nonumber \quad\quad\quad\quad\quad +(1-2\alpha)(c_n+c_{n+1})+(\alpha ^2 -1)a_n  \Big]B_n +2(b_n -\alpha b_{n+1} )C_{n+1}\\
&\nonumber \quad \quad\quad\quad\quad +2(b_n -\alpha b_{n-1})C_{n}\;.
\end{align}
In addition, the following relations hold:
\begin{align}
&b_n =\alpha_n\;,\quad c_n =\left(\alpha_n -\alpha_{n-1}  \right)\sum_{j=0} ^{n-1} B_j\;,\quad\quad~ ~~\mbox{\rm if}~~ a=0,~c=-1\;, \label{power-pi1}\\
&a_n=\alpha_n\;,\quad b_n =-\alpha_n c+\left(\alpha_n-\alpha_{n-1} \right)\sum_{j=0} ^{n-1} B_j \;,\quad ~~\mbox{\rm if}~~a=1\;. \label{power-pi2}
\end{align}
\end{lemma}
}
\begin{proof}
Let $(P_n)_{n\geq0}$ be a monic OPS satisfying \eqref{equation-explicit-to solve-general}.
Applying the operator $\mathcal{D}_q$ to both sides of (\ref{TTRR}) and using \eqref{def-Dx-fg}, we deduce
\begin{align*}
\mathcal{S}_q P_n(x) =- \alpha x\mathcal{D}_q P_n(x) +\mathcal{D}_q P_{n+1}(x) +B_n \mathcal{D}_q P_n(x) +C_n \mathcal{D}_q P_{n-1}(x)\;.
\end{align*}
Multiplying both sides of this equality by $(ax-c)\mathtt{U}_2(x)$ and then using successively (\ref{equation-explicit-to solve-general}), (\ref{equation-alternative-with Dx}), and (\ref{TTRR}), we obtain
a vanishing linear combination of the polynomials $P_{n+3}, P_{n+2}, \dots,P_{n-3}$.
Thus, setting
$t_n=c_n/C_n$ for $n=1,2,3,\ldots$,  
after straightforward computations we obtain \eqref{eq1S-eq2S} together with the following equations:
\begin{align}
&\label{seq3S} (a_{n+1}-a_{n+2})B_{n+1} +(a_n -a_{n-1})B_n +b_{n+2}-2\alpha b_{n+1}+b_n =0 , \\
&\label{seq4S} (a_{n+1}-a_{n+2}-t_{n+2})B_{n+1} +(a_{n}-a_{n-1}+t_{n+1}+t_{n})B_n -t_{n-1} B_{n-1} \\ 
&\nonumber  \quad\quad  +b_{n+1}-2\alpha b_{n}+b_{n-1} =0 , \\
&\label{seq5S} (a_{n+1}-a_{n+2})B_{n+1}^2+2(1-\alpha)a_nB_n^2+(a_n-a_{n-1})B_n B_{n+1}+ (a_n -a_{n+2})C_{n+1}  \\ 
&\nonumber \quad +(b_{n+1}+b_n-2\alpha b_{n+1})B_{n+1} +(b_{n+1}+b_n -2\alpha b_{n} )B_n  +(a_n - a_{n-2})C_n\\ 
&\nonumber \quad +c_{n+2}-2\alpha c_{n+1}+c_n  = (1-\alpha ^2)a_n ,\\
&\label{seq6S} (2(1-\alpha)a_n+t_n)B_{n}^2 + (t_n +a_{n-1}-a_{n-2})B_{n-1}^2+(b_{n} +b_{n-1} -2\alpha b_{n})B_{n} \\ 
&\nonumber \quad +(a_n-t_{n-1}-t_{n+1}-a_{n+1})B_n B_{n-1} +(b_{n-1}+b_n -2\alpha b_{n-1} )B_{n-1} \\ 
&\nonumber \quad +(a_n-a_{n+2}-t_{n+2}-t_{n+1})C_{n+1}+( 2(1+\alpha)t_n+a_n-a_{n-2})C_n  \\
&\nonumber \quad -(t_{n-2}+t_{n-1})C_{n-1} +c_{n+1}-2\alpha c_{n}+c_{n-1} = (1-\alpha^2)(t_n+a_n)  ,\\
&\label{seq7S} 2(1-\alpha)a_nB_n^3 +2(1-\alpha)b_nB_{n}^2 +\left[ (2a_n -a_{n+2}-a_{n-1})C_{n+1}\right.\\
&\nonumber \quad \left. + (2a_n -a_{n+1}-a_{n-2})C_n  + c_{n+1}-2\alpha c_n +c_n -2\alpha c_{n+1}  - (1-\alpha^2) a_n \right] B_{n} \\ 
&\nonumber \quad +(c_{n+1}+a_{n+1}C_{n+1}-a_{n+2}C_{n+1})B_{n+1} +(c_n +a_{n-1}C_{n}-a_{n-2}C_n)B_{n-1} \\ 
&\nonumber \quad +2(b_n -\alpha b_{n+1})C_{n+1} +2(b_n -\alpha b_{n-1} )C_n   = (1 -\alpha^2)b_n .
\end{align}

\eqref{eq3S} (respectively, \eqref{eq4S}) is obtained by shifting $n$ to $n+1$ in \eqref{seq4S} (respectively, \eqref{seq6S}) and combining it with \eqref{seq3S} (respectively, \eqref{seq5S}) and by using \eqref{eq1S-eq2S}. \eqref{eq5S} follows from \eqref{eq1S-eq2S} and \eqref{seq7S}. Now suppose that $a=1$. Using \eqref{TTRR}, we may write $ P_n (x)=x^n -x^{n-1}\sum_{j=0} ^{n-1} B_j +w_nx^{n-2}+\cdots$, for some complex sequence $(w_n)_{n\geq0}$. Using \eqref{Dx-xn}, we compare the two first coefficients of higher power of $n$ in both side of \eqref{equation-explicit-to solve-general} to deduce \eqref{power-pi2}. \eqref{power-pi1} is obtained in a similar way and this completes the proof.  
\end{proof}

\begin{remark}
According to the previous lemma, the coefficients $B_n$ and $C_n$ of the TTRR \eqref{TTRR} of any monic OPS $(P_n)_{n\geq 0}$ fulfilling \eqref{equation-explicit-to solve-general} must fulfill \eqref{eq1S-eq2S}--\eqref{eq5S}. However, for each concrete polynomial $\pi(x)=ax-c$ appearing in \eqref{equation-explicit-to solve-general}, we need to take into account some initial conditions which will be specified in the proof of the main result in all situations according to the degree of $\pi$. 
Indeed, for instance, it is clear that
$$B_n=0,\quad \quad C_{n+1}=1/4\quad (n=0,1,2,\ldots)\;,$$ provide a solution of the system \eqref{eq1S-eq2S}--\eqref{eq5S}. The corresponding monic OPS is
$$P_n (x)=U_n \left(x\right)~\quad~~(n=0,1,2,\ldots)\;, $$ where $(U_n)_{n \geq 0}$ is the monic Chebyschev polynomials of the second kind. However this sequence $(P_n)_{n \geq 0}$ does not provide a solution of \eqref{equation-explicit-to solve-general}  (see below \eqref{initialcondition-C1-firstcase} for the case $\pi(x)=1$ and \eqref{why-chebyshev-first-kind} for the case $\pi(x)=x-c$). 
\end{remark}

 The system of equations \eqref{eq1S-eq2S}--\eqref{eq5S} is non-linear and so, in general it is not easy to solve. Nevertheless, the same system was solved in \cite{KCDMJP2021c} for the case where $(P_n)_{n\geq 0}$ is a classical orthogonal polynomial and so, it was possible among this known class of orthogonal polynomials to find those which satisfy \eqref{eq1S-eq2S}--\eqref{eq5S}. For the present case the task is more harder because we do not have such information and even the initial conditions are different. We will see that some patterns appear associated with the system of equations \eqref{eq1S-eq2S}--\eqref{eq5S} which will allow us to solve the system for each possible case of the degree of the polynomial $\pi$. 

Recall that from \eqref{eq1S-eq2S}, we have 
\begin{align}\label{general-tn-struc}
t_n = \frac{c_n}{C_n} =k_1 q^{n/2}+k_2q^{-n/2}\quad (n=1,2,3,\ldots)\;,
\end{align} 
where $k_1$ and $k_2$ are two complex numbers. Setting $c_0=C_0=0$, we define $$t_0:=k_1+k_2\;,$$ by compatibility with \eqref{general-tn-struc}. Hereafter we assume that $$r_n=t_n+a_n-a_{n-1}\neq 0\;\quad (n=0,1,2\ldots)\;,$$ where $r_n$ is defined in \eqref{eq3S}.

\subsection{Case $\pi(x)=1$}
In this case, \eqref{equation-explicit-to solve-general} can be rewritten as
\begin{align}
\mathcal{S}_q P_n(x)=\alpha_nP_n(x)+c_nP_{n-1}(x)\quad (n=0,1,2,\ldots)\;. \label{case-pi-1}
\end{align}  
  
\begin{lemma}\label{uniqueness-case-1}
Let $(P_n)_{n\geq0}$ be a monic OPS satisfying \eqref{case-pi-1}. Then 
\begin{align}
\left( c_2 C_1-q^{-1/2}c_1C_2\right) \left(c_2C_1-q^{1/2}c_1C_2\right)c_1  = 0\;. \label{uniqueness-cond-case-1}
\end{align}
\end{lemma}

\begin{proof}
Since $(P_n)_{n\geq0}$ satisfies \eqref{case-pi-1}, then $a_n =0$, for each $n=0,1,2,\ldots$, and by \eqref{eq1S-eq2S}  and \eqref{power-pi1}, we obtain $c_1=(\alpha-1)B_0$ and
\begin{align}
t_n =k_1q^{n/2}+k_2q^{-n/2},~ k_1 = \frac{c_2 C_1-q^{-1/2}c_1C_2}{(q-1)C_1C_2},~ k_2= \frac{c_2C_1-q^{1/2}c_1C_2}{(q^{-1}-1)C_1C_2} .\label{tn-pi-1}
\end{align}
Suppose, contrary to our claim, that \eqref{uniqueness-cond-case-1} does not hold. This means that $k_1k_2B_0\neq 0$. We claim that the following relations hold.
\begin{align}
&(\alpha+1)(2C_1-1)=\big(B_1-(2\alpha+1)B_0\big)B_0\;,\label{initialcondition-C1-firstcase}\\
&(2\alpha+1)B_0C_2 -B_0(B_0+B_1)B_2=\frac{1}{2}(B_0+B_1)\left(2\alpha+1 -u\frac{B_0+B_1}{\alpha+1}B_0  \right)\;,\label{initialcondit-B2-C2-firstcase}\\
&\alpha(\alpha+1)(4C_2-1)-(2\alpha+1)(B_0+B_1)B_2=\big(B_0-uB_1\big)(B_0+B_1)\;,\label{initialcondit-B2-firstcase}
\end{align}
where $u=4\alpha^2+2\alpha-1$. Indeed, taking $n=2$ in \eqref{case-pi-1}, and using \eqref{TTRR} and \eqref{Dx-xn}, we obtain $B_0B_1-C_1+1-\alpha^2=b_2(B_1B_0-C_1)-c_2B_0$ and therefore \eqref{initialcondition-C1-firstcase} holds using \eqref{power-pi1}. Similarly, taking $n=3$ in \eqref{case-pi-1}, we obtain the following equations.
\begin{align*}
\alpha(\alpha+1)(4C_1+4C_2-3)&=(2\alpha+1)(B_0B_1+B_0B_2+B_1B_2) \\
&-(4\alpha^2+2\alpha-1)(B_0^2+B_0B_1+B_1^2)\;,\\
(2\alpha+1)^2(B_0C_2+B_2C_1)&=(\alpha+1)(2B_0B_1B_2+B_0+B_1+B_2)\\
&+(4\alpha^2+2\alpha-1)\big((B_0+B_1+B_2)C_1-(B_0+B_1)B_0B_1\big)\;.
\end{align*}
Thus \eqref{initialcondit-B2-C2-firstcase}--\eqref{initialcondit-B2-firstcase} follow from these equations using \eqref{initialcondition-C1-firstcase}.
Set $r=-k_1/k_2$. We write \eqref{eq3S} as $t_{n+3}B_{n+2}-t_{n+1}B_{n+1}=t_{n+2}B_{n+1}-t_{n}B_{n}$ and proceeding in a recurrent way, we have 
\begin{align}\label{explicit-Bn-firstcase}
B_n =\frac{(1-r)(1-rq)B_0q^{n/2} +K_b(1-q^{n/2})(1-rq^{(n+1)/2})}{(1-rq^n)(1-rq^{n+1})}q^{n/2},
\end{align}
where $K_b=\big((1-r)B_0-(1-rq^2)q^{-1}B_1\big)/(1-q^{-1/2})$. Using \eqref{power-pi1}, we then deduce
\begin{align}\label{explicit-c_n-firstcase}
c_n= \frac{(1-q^{n-1/2})(1-q^{n/2})\big(B_0(1-rq)(1+q^{n/2})+K_b(q^{1/2}-q^{n/2}) \big)q^{-n/2}}{2(1+q^{1/2})(1-rq^n)}\;.
\end{align}
Also since $c_n=t_nC_n$, we obtain
\begin{align}\label{explicit-C_n-firstcase}
C_n= \frac{(1-q^{n-1/2})(1-q^{n/2})\big(B_0(1-rq)(1+q^{n/2})+K_b(q^{1/2}-q^{n/2}) \big)}{2k_2(1+q^{1/2})(1-rq^n)^2}\;.
\end{align}
Since $0<q<1$ and $t_n=k_2(1-rq^n)q^{-n/2}$, we obtain
\begin{align*}
\lim_{n\rightarrow \infty} B_n=0\;, \quad \lim_{n\rightarrow \infty} C_n=\frac{B_0(1-rq)+K_bq^{1/2}}{2k_2(1+q^{1/2})}\;.
\end{align*}
On the other hand, after rewriting \eqref{eq4S} as
\begin{align*}
(q^{-1}&+q^{-1/2})(1-rq^{n+3/2})(C_{n+1}-1/4)-2(1+\alpha)(1-rq^n)(C_n-1/4)\\
&+(q+q^{1/2})(1-rq^{n-3/2})(C_{n-1}-1/4)=(1-rq^n)(B_n^2 -2\alpha B_nB_{n-1}+B_{n-1} ^2)\;,
\end{align*}
we take the limit and obtain $$\lim_{n\rightarrow \infty}C_n=1/4\;.$$
The equality between the two limits obtained on $C_n$ requires 
$$k_2=2\frac{B_0(1-rq)+K_bq^{1/2}}{1+q^{1/2}}\;,$$ and consequently \eqref{explicit-C_n-firstcase} becomes
\begin{align*}
C_n=\frac{(1-q^{n-1/2})(1-q^{n/2})(1+\mathfrak{a}q^{n/2})}{4(1-rq^n)^2}\;,\quad \mathfrak{a}=\frac{B_0(1-rq)-K_b}{B_0(1-rq)+K_bq^{1/2}}\;.
\end{align*}
Further we obtain
$$\lim_{n\rightarrow \infty}q^{-n/2}(C_n-1/4)=\frac{\mathfrak{a}-1}{4}\;,\quad \mathfrak{b}=\lim_{n\rightarrow \infty}c_nB_n=\frac{K_b\big(B_0(1-rq)+K_bq^{1/2}\big)}{2(1+q^{1/2})} \;.$$
We now rewrite \eqref{eq5S} as
\begin{align*}
2(1-\alpha)\alpha_nB_n^2 &+c_{n+1}B_{n+1}+c_nB_{n-1} +(1-2\alpha)(c_n+c_{n+1})B_n\\ &+2(1-\alpha^2)\gamma_{n+1}(C_{n+1}-1/4) -2(1-\alpha^2)\gamma_{n-1}(C_n-1/4)=0\;. 
\end{align*}
Taking the limit to the above expression yields $\mathfrak{a}=1-8\mathfrak{b}q^{-1/2}$. This means
\begin{align}
K_b\Big(q^{1/2}(1+q^{1/2})^2-4\big(B_0(1-rq)+K_bq^{1/2} \big)^2  \Big)=0\;.
\end{align}
We distinguish two cases.
\begin{itemize}
\item[i-]If $K_b=0$, then $\mathfrak{a}=1$ from what is preceding, we obtain
\begin{align*}
&B_n=\frac{B_0(1-rq)(1-r)q^n}{(1-rq^n)(1-rq^{n+1})}\;,~C_n=\frac{(1-q^{n-1/2})(1-q^n)}{4(1-rq^n)^2}\;,\\
&c_n=\frac{B_0(1-rq)(1-q^{n-1/2})(1-q^n)q^{-n/2}}{2(1+q^{1/2})(1-rq^n)}\;,~k_2=\frac{B_0(1-rq)}{2(1+q^{1/2})}\;.
\end{align*}
This means that $B_0$ and $r$ are the only possible free parameters. Using the above equations, \eqref{initialcondition-C1-firstcase}--\eqref{initialcondit-B2-firstcase} become
\begin{align*}
2\Big(\frac{\alpha(1-rq)q^{1/2}}{(\alpha+1)(1-rq^2)}  -1\Big)B_0^2 =\frac{(q-1)(q^{1/2}-1)}{2(1-rq)^2}-1 \;,\quad \quad\quad \quad\quad\\
\frac{1-rq}{\alpha+1/2}q^{1/2}\Big(\frac{(1-r)q^{3/2}}{1-rq^3}  -\frac{\alpha(4\alpha^2+2\alpha-1)}{\alpha+1}\Big)B_0^2 =\frac{(1-q^{3/2})(1-q)^2}{2(1-rq)}-1+rq^2& \;,\\
\frac{2q(2\alpha+1)(1-rq)^2}{(\alpha+1)(1-rq^2)^2}\Big(\frac{(1-r)q^{3/2}}{1-rq^3}+\frac{2\alpha}{2\alpha+1}-2\alpha\frac{(1-r)q^{1/2}}{1-rq}  \Big)B_0^2 =\frac{(1-q^{3/2})(1-q^2)}{(1-rq^2)^2}&-1  \;.
\end{align*}
Since by assumption $r\neq 0$, then it is not hard to see that these equations are incompatible.
\item[ii-] If $(B_0(1-rq)+K_bq^{1/2})^2 =q^{1/2}(1+q^{1/2})^2/4$, then we proceed exactly as in the previous case to see that this is also impossible.
Thus the result follows.

\end{itemize}

\end{proof}

\section{Main results}
We are ready to state and prove our first result from which we will recover the counterexample to Conjecture \ref{IsmailConj2478} presented in \cite{KCDM2022}. 

\begin{theorem}\label{T1}
The Chebyshev polynomials of the first kind, the Al-Salam Chihara polynomials with nonzero parameters $c$ and $d$ such that $(c,d)=\pm (1,q^{1/2})$ and the continuous dual $q$-Hahn polynomials with $q$ replaced by $q^{1/2}$ and $ab=1$, $ac=q^{1/4}$, and $bc=-q^{1/4}$ are the only OPS  satisfying \eqref{case-pi-1}. 
\end{theorem}

\begin{proof}
Note that \eqref{uniqueness-cond-case-1} is equivalent to $k_1k_2B_0= 0$. We then have the following cases.\\
\textbf{I- }{\bf Case $B_0=0$.}\\
For this this case using \eqref{initialcondition-C1-firstcase}--\eqref{initialcondit-B2-firstcase}, we obtain $B_1=0$, $C_1=1/2$ and $C_2=1/4$. This implies that $c_1=0=c_2$ and $t_1=0=t_2$ using \eqref{power-pi1} and \eqref{eq1S-eq2S}. Hence $t_n=0=c_n$ for all $n=1,2,\ldots$. We deduce $B_n=0\;.$
With this \eqref{eq5S} reads as $ \gamma_{n+1}(C_{n+1}-1/4) -\gamma_{n-1}(C_n-1/4)=0$. Hence we obtain
$$B_{n-1}=0\;,\quad C_1=1/2\;,\quad C_{n+1}=1/4 \quad (n=1,2,\ldots)\;.$$ 
This is the Chebyshev polynomial of the first kind, $(T_n)_{n\geq 0}$, and we have 
\begin{align}
\mathcal{S}_qT_n(x)=\alpha_nT_n(x)\quad (n=0,1,2,\ldots)\;.\label{full-equation-chebyshev-case1}
\end{align}
\\
\textbf{II- }{\bf Case $k_1=0$.}\\
This means $c_2C_1-q^{-1/2}c_1C_2=0$. We assume here that $B_0\neq 0$ if not we obtain again the solution of the previous case. Then from \eqref{initialcondit-B2-C2-firstcase}, we see that $B_0+B_1\neq 0$. Using \eqref{power-pi1} and \eqref{initialcondition-C1-firstcase}, we then obtain, from $k_1=0$, the following equation
\begin{align}\label{B_0C_2-firstcase-k1=0}
B_0C_2= (2\alpha+1)q^{1/2}(B_0+B_1)\Big(\frac{1}{2}+\frac{B_1-(2\alpha+1)B_0}{2(\alpha+1)}B_0  \Big)\;.
\end{align}
Using \eqref{B_0C_2-firstcase-k1=0}, \eqref{initialcondit-B2-C2-firstcase} becomes
\begin{align}\label{B_0B_2-firstcase-k_1=0}
B_0B_2 =\frac{1}{2}(2\alpha+1)(1+q^{1/2})q^{1/2}+\frac{B_0}{2(\alpha+1)}&\Big[ \big(4\alpha^2+2\alpha-1+(2\alpha+1)^2q^{1/2} \big)B_1    \\
&+\big(4\alpha^2+2\alpha-1-(2\alpha+1)^3q^{1/2} \big)B_0  \Big]\;.\nonumber
\end{align}
In the meantime, we claim that 
\begin{align}\label{B_nfirstcase-k_1=0}
B_n=\frac{1}{q^{1/2}(q^{1/2}-1)}\left((B_1-q^{1/2}B_0)q^n +(qB_0-B_1)q^{n/2}  \right)\quad (n=0,1,2,\ldots)\;.
\end{align}
Indeed \eqref{eq3S} in this case reads as $q^{-1/2}B_{n+1}+(1+q^{1/2})B_{n+1}+qB_n=0$. Since $q$ and $q^{1/2}$ are the solutions for the associated characteristic equation, we obtain $B_n=vq^n +rq^{n/2}$, for some $v,r\in \mathbb{C}$, and therefore \eqref{B_nfirstcase-k_1=0} holds by writing $B_0$ and $B_1$ in term of $r$ and $v$ after taking $n=0$ and $n=1$ in the above solution.   

We also claim that the following equations hold.
\begin{align}
&1+\alpha +q^{-1}B_0B_1 -(1+2q^{-1/2})B_0^2=0\;,\label{equation1-firstcase-k1=0}\\
&B_0(B_0+B_1)\Big( (q^{3/2}+3q+4q^{1/2}+5+2q^{-1/2}+q^{-1})B_0- (q^{1/2}+1+2q^{-1/2}+q^{-1})B_1 \Big) \label{equation2-firstcase-k1=0} \\
&\quad \quad \quad \quad \quad \quad \quad \quad\quad \quad \quad \quad\quad =-\alpha(\alpha+1)B_0+2\alpha(2\alpha+1)(\alpha+1)q^{1/2}(B_0+B_1)\;.\nonumber
\end{align}

Indeed \eqref{equation1-firstcase-k1=0} is obtained using the expression of $B_2$ computed from \eqref{B_nfirstcase-k_1=0} and replacing it into \eqref{B_0B_2-firstcase-k_1=0}. \eqref{equation2-firstcase-k1=0} is also obtained by replacing in \eqref{initialcondit-B2-firstcase} the expression of $B_2$ and $C_2$ computed from \eqref{B_nfirstcase-k_1=0} and \eqref{B_0C_2-firstcase-k1=0}, respectively. 

Solving the system of equations \eqref{equation1-firstcase-k1=0}--\eqref{equation2-firstcase-k1=0}, we obtain the following equation
$$\big(B_0^2-q^{1/2}\big)\big(4B_0^2 -(1+q^{1/2})^2\big)=0\;,$$
and we then consider two cases.
\begin{itemize}
\item[i-] Suppose that $B_0^2=\mbox{$\frac{1}{4}$}(1+q^{1/2})^2$.\\
Then we obtain from \eqref{equation1-firstcase-k1=0}, $B_1=qB_0$ and consequently $B_n=B_0q^n$ from \eqref{B_nfirstcase-k_1=0}. Hereafter we denote $$s=\pm 1\;.$$
We also obtain from \eqref{initialcondition-C1-firstcase}, $C_1=(1-q)(1-q^{1/2})/4$. From \eqref{power-pi1} we obtain $c_1=(\alpha-1)B_0$ and since $k_1=0$, we have $k_2=q^{1/2}c_1/C_1= 4s$. Now from \eqref{power-pi1}, we compute $c_n$ and from $t_n=k_2q^{-n/2}=c_n/C_n$, we deduce $C_n$.
This gives
\begin{align}\label{Bn-Cn-solution-Al-Salam-Chihara-firstcase-k1=0}
B_n=\frac{s}{2}(1+q^{1/2})q^n\;, ~ C_n=\frac{1}{4}(1-q^n)(1-q^{n-1/2})\;,~
c_n=\frac{s(1-q^n)(1-q^{n-1/2})}{4q^{n/2}}\;.
\end{align}
In addition equations \eqref{eq4S}--\eqref{eq5S} become
\begin{align}
(q^{-1}+q^{-1/2})&(C_{n+1}-1/4)-2(1+\alpha)(C_n-1/4)+(q+q^{1/2})(C_{n-1}-1/4)\label{Cn-equation-when-k1=0}\\
&=B_n^2-2\alpha B_nB_{n-1}+B_{n-1}^2\;,\nonumber\\
2(1-\alpha)&\alpha_nB_n^2 +c_{n+1}B_{n+1} +c_nB_{n-1}+(1-2\alpha)(c_n+c_{n+1})B_n \label{last-equation-firstcase-when-k1=0}\\ &+2(1-\alpha^2)\gamma_{n+1}(C_{n+1}-1/4)-2(1-\alpha^2)\gamma_{n-1}(C_n-1/4)=0\;.\nonumber
\end{align}
One may easily check that the above solution given in \eqref{Bn-Cn-solution-Al-Salam-Chihara-firstcase-k1=0} satisfy \eqref{Cn-equation-when-k1=0}--\eqref{last-equation-firstcase-when-k1=0}. This means that
\begin{align}\label{full-solution-Al-salam-Chihara}
\mathcal{S}_qQ_n(x;s,sq^{1/2}|q)=\alpha_nQ_n(x;s,sq^{1/2}|q)
+\frac{s(1-q^n)(1-q^{n-1/2})}{4q^{n/2}}Q_{n-1}(x;s,sq^{1/2}|q) \;,
\end{align}
where $(Q_n(.;c,d|q))_{n\geq 0}$ is the monic Al-Salam Chihara polynomial.
 
\item[ii-] Suppose that $B_0^2=q^{1/2}$.\\
Proceeding exactly as in (i), we obtain 
\begin{align}
B_n&=\frac{s}{2}\Big((1+q^{-1/2})q^{n/2} +1-q^{-1/2}\Big)q^{(2n+1)/4}\;,\label{Bn-solution-continuous-dual-q-hahn}\\
C_n&=\frac{1}{4}(1+q^{(n-1)/2})(1-q^{n/2})(1-q^{n-1/2})\;,\label{Cn-solution-continuous-dual-qhahn}\\
c_n&=\frac{s}{4}(1-q^{n/2})(1-q^{n-1/2})(1+q^{(n-1)/2})q^{-(2n-1)/4}\;.\label{cn-solution-continuous-dual-qhahn}
\end{align}
These coefficients also satisfy \eqref{Cn-equation-when-k1=0}--\eqref{last-equation-firstcase-when-k1=0}. Thus 
\begin{align}
\mathcal{S}_qH_n(&x;s,-s,sq^{1/4}|q^{1/2})=\alpha_nH_n(x;s,-s,sq^{1/4}|q^{1/2})\label{full-equation-for-continuousdual-qhahn-firstcase}\\
&+\frac{s(1-q^{n/2})(1-q^{n-1/2})(1+q^{(n-1)/2})}{4q^{(2n-1)/4}}H_{n-1}(x;s,-s,sq^{1/4}|q^{1/2}) \;,\nonumber
\end{align}
where $(H_n(.;a,b,c|q))_{n\geq 0}$ is the monic continuous dual $q$-Hahn polynomial.
\end{itemize}

Similarly, for the case $k_2=0$, we obtain the same solutions with $q$ replaced by $q^{-1}$. Theorem \ref{T1} follows from this and Lemma \ref{uniqueness-case-1}.  
\end{proof}
\begin{remark}\label{remark-intro}
In \cite[p.301]{IsmailBook2005}, it is proved that 
$$\mathcal{D}_qT_n(x)=\gamma_n U_{n-1}(x) \quad (n=0,1,\ldots)\;.$$
This relation between the Chebyshev polynomials of first and second kind is useful when solving problems related with the Askey-Wilson operator. Equation \eqref{full-equation-chebyshev-case1} derived here also shows that the Chebyshev polynomials of the first kind constitute an appropriate basis for the averaging operator $\mathcal{S}_q$. This can be used to solve problems related with the mentioned operators as well as some connection formulae and linearisation problems. For simplicity, in the next case we will make some assumptions on the coefficients of the TTRR \eqref{TTRR} satisfied by the OPS solutions of \eqref{equation-explicit-to solve-general}.
\end{remark}

\subsection*{Case $\deg \pi =1$}
In this case, \eqref{equation-explicit-to solve-general} becomes 
\begin{align}\label{lastcase-deg-pi-is-2}
(x-c)\mathcal{S}_qP_n(x)=(\alpha_nx+b_n)P_n(x)+c_nP_{n-1}(x)\;,~c\in \mathbb{C}\;.
\end{align}
The following result holds.

\begin{theorem}\label{T2}
Let $(P_n)_{n\geq 0}$ be a monic OPS satisfying \eqref{TTRR} and \eqref{lastcase-deg-pi-is-2}. Assume that 
\begin{align}
\lim_{n \rightarrow \infty} q^{-n/2}B_n=0\;,~\lim_{n \rightarrow \infty} q^{-n/2}(C_n-1/4)=0\;.\label{assumptions-second-case}
\end{align}
 Then $(P_n)_{n\geq 0}$ are the Chebyschev polynomials of the first kind. 
\end{theorem}
\begin{proof}
Taking successively $n=1$ and $n=2$ in \eqref{lastcase-deg-pi-is-2} using \eqref{TTRR} and \eqref{Dx-xn} we obtain the following:
\begin{align}
&b_1=-\alpha c+(\alpha-1)B_0\;,\quad c_1=(\alpha-1)(B_0-c)B_0\;, \label{b1-pi-2-case-2}\\
&b_2=-(2\alpha^2-1)c+(\alpha -1)(2\alpha+1)(B_0+B_1)\;,\\ 
&c_2=2(\alpha^2-1)(C_1-B_0B_1-1/2)+(\alpha-1)(2\alpha+1)(B_0+B_1)(B_0+B_1-c)\;,\label{c2-lastcase}\\
&(b_2+c)(B_0B_1-C_1)+(1-\alpha^2)c=c_2B_0\;. \label{b2-pi-2-case-2}
\end{align}
Since $0\neq r_n=t_n+\alpha_n-\alpha_{n-1}=\widehat{a}q^{n/2}+\widehat{b}q^{-n/2}$, without lost of generality, let's assume that $\widehat{b}\neq 0$. We then obtain $$r_n=\widehat{b}q^{-n/2}(1-rq^n),\;~r=-\widehat{a}/\widehat{b}\;.$$ 
Solving \eqref{eq3S} we find
\begin{align*}
B_n =\frac{(1-r)(1-rq)B_0q^{n/2} +\widehat{K}_b(1-q^{n/2})(1-rq^{(n+1)/2})}{(1-r q^{n})(1-rq^{n+1})}q^{n/2}\; ,
\end{align*}
for $n=0,1,2,\ldots$, where $\widehat{K}_b=\big((1-r)B_0 -q^{-1}(1-rq^2)B_1\big)/(1-q^{-1/2})$. From \eqref{assumptions-second-case}, we obtain $\widehat{K}_b=\lim_{n\rightarrow \infty}q^{-n/2}B_n=0$. Hence
\begin{align}\label{first-Bn-pi-2-a}
B_n =\frac{(1-r)(1-rq)B_0q^{n}}{(1-r q^{n})(1-rq^{n+1})}\; ,
\end{align}
for each $n=0,1,\ldots$. We distinguish three cases.\\
\textbf{Case 1} Assume that $B_0=0$.\\
This implies that $B_n=0$ and solving \eqref{eq4S}, we obtain
\begin{align}\label{first-Cn-pi-2-b-homo}
C_{n+1} =\frac{1}{4}+\frac{(1-rq^{1/2})(1-rq^{3/2})(C_1-1/4)q^{n/2} +\widehat{K}_c(1-q^{n/2})(1-rq^{(n+2)/2})}{(1-r q^{n+1/2})(1-rq^{n+3/2})}q^{n/2}\; ,
\end{align}
where $\widehat{K}_c=\Big((1-rq^{1/2})(C_1-1/4) -q^{-1}(1-rq^{5/2})(C_2-1/4)\Big)/(1-q^{-1/2})$.\\
We claim that $$C_1=1/2\;,C_2=1/4\;.$$
Indeed, from \eqref{b1-pi-2-case-2}--\eqref{b2-pi-2-case-2}, we obtain $c_1=0$, $c_2=2(\alpha^2-1)(C_1-1/2)$, $(C_1-1/2)c=0$. For $n=3$ in \eqref{lastcase-deg-pi-is-2} taking into account that $B_n=0$, we also find $$b_3=(\alpha_3-\alpha_2)(B_0+B_1+B_2)-\alpha_3 c=-\alpha(4\alpha^2-3)c\;,$$
\begin{align*}
\quad\quad c_3&=- 4\alpha(\alpha^2-1)(B_0B_1+B_0B_2+B_1B_2-C_1-C_2+3/4)\\
&+(b_3+\alpha_2c)(B_0+B_1+B_2)\\
&=  4\alpha(\alpha^2-1)(C_1+C_2-3/4)\;,
\end{align*}
and
$$(C_1-B_0B_1)c_3+(1-\alpha^2)(B_0+B_1+B_2)c=(b_3+c)(B_0C_2+B_2C_1-B_0B_1B_2)\;,$$
\begin{align*}
(b_3+\alpha c)(B_0B_1+B_0B_2&+B_1B_2-C_1-C_2)+(\alpha_3-1)(B_0C_2+B_2C_1-B_0B_1B_2)\\
&=(\alpha^2-1)(B_0+B_1+B_2)+(B_0+B_1)c_3+3\alpha(\alpha^2-1)c\;.
\end{align*}
From these equations, we obtain
$$c_3=4\alpha(\alpha^2-1)(C_1+C_2-3/4)=0\;.$$ i.e.
\begin{align}\label{why-chebyshev-first-kind}
C_1+C_2-\frac{3}{4}=0\;.
\end{align}
Since $t_n=c_nC_n=2\alpha t_{n-1}-t_{n-2}$, from $c_1=c_3=0$, we deduce $c_n=0=t_n$ for each $n=1,2,\ldots$. Then $C_1=1/2$ and the claim holds by \eqref{why-chebyshev-first-kind}. This also implies that $$r_n=\alpha_n-\alpha_{n-1}=\frac{1}{2}(1-q^{1/2})(1-q^{n-1/2})\;.$$
That is $\widehat{b}=\frac{1}{2}(1-q^{1/2})$ and $r=q^{-1/2}$. 
Therefore we obtain with \eqref{first-Cn-pi-2-b-homo} that
$$B_{n}=0\;,~C_1=1/2\;,~C_{n+2}=1/4\;,~b_n=-\alpha_nc,\quad~c_n=0,~ n=0,1,\ldots\;.$$
are solutions of \eqref{lastcase-deg-pi-is-2} since \eqref{eq5S} is also satisfied. This is the monic Cheybyshev polynomial of the first kind and so
\begin{align}\label{relation-case2-chebyshev}
(x-c)\mathcal{S}_qT_n(x)=\alpha_n\big(T_{n+1}(x)-c T_n(x)+\mbox{$\frac{1}{4}$}T_{n-1}(x)  \big)~~\quad~ (n=0,1,\ldots)\;.
\end{align}
\textbf{Case 2} Assume now that $r=0$.\\
In this case $\widehat{a}=0$ and so \eqref{first-Bn-pi-2-a} becomes $B_n=B_0q^n$. In the meantime, \eqref{eq4S} reduces to
\begin{align*}
q^{-1}(1+q^{1/2})(C_{n+1}-1/4)-2(1+\alpha)(C_n-1/4)&+q(1+q^{-1/2})(C_{n-1}-1/4)\\
&=2(\alpha-1)(2\alpha+1)B_0 ^2q^{2n-1} \;.
\end{align*}
Solutions of this equation are given by the following expression $$C_{n+1}=\frac{B_0 ^2}{2(\alpha+1)}q^{2n+1}+vq^n+uq^{n/2}+\frac{1}{4},\;\quad u,v\in \mathbb{C}\;.$$ Taking into account our assumption \eqref{assumptions-second-case}, we find $u=0$. Let $a$ and $b$ be two complex numbers such that $a+b=2B_0$ and $ab=1+4C_1/(q-1)$. Then we find
$$B_n=\frac{1}{2}(a+b)q^n,~~C_{n+1}=\frac{1}{4}(1-abq^n)(1-q^{n+1}),~~~~\quad~ a^2+b^2=2\alpha ab\;.$$
We claim that this does not provide a solution to \eqref{lastcase-deg-pi-is-2}. Indeed, since we have $r_n=t_n+\alpha_n-\alpha_{n-1}=\widehat{b}q^{-n/2}$, initial conditions \eqref{b1-pi-2-case-2}--\eqref{b2-pi-2-case-2} become
\begin{align}
&\widehat{b}q^{-1/2}C_1 +(\alpha-1)B_0c=(\alpha-1)(B_0 ^2 +C_1),\label{solve-init1}\\
&(\alpha-1)^{-1}~\widehat{b}q^{-1}B_0C_2 -2(\alpha+1)(C_1-B_0B_1-1/2)c= T,\\
&(\alpha-1)^{-1}~\widehat{b}q^{-1}C_2 +(2\alpha+1)(B_0+B_1)c =U\;,\label{solve-init3}
\end{align}
with $U$ and $T$ given by $U=2(\alpha+1)(C_1-B_0B_1-1/2)+(2\alpha+1)(C_2+(B_0+B_1)^2)$ and $T=(2\alpha+1)\Big((B_0B_1-C_1)(B_0+B_1) +B_0C_2\Big)$. Taking into account that $a^2+b^2=2\alpha ab$ i.e. $b=aq^{\pm 1/2}$, we see that the above system \eqref{solve-init1}--\eqref{solve-init3} has solutions in $\widehat{b}$ and $c$ if and only if $a$ is a solution of the following equation $q^2a^4-q^{1/2}(1+q)a^2+1=0$. This means $a=\pm q^{-1/4}$ or $a=\pm q^{-3/4}$. But for any choice of $a$, we obtain $C_1C_2=0$ and this is impossible. Hence the above system has no solutions.

\textbf{Case 3} Assume that $r\neq 0$.\\
Define two complex numbers $a$ and $b$ such that 
$$ab=r \;, ~~\quad  a-b=\frac{2(1-rq)B_0}{q^{1/4}(1+q^{1/2})}\;.$$
Then \eqref{first-Bn-pi-2-a} becomes
\begin{align*}
B_n=q^{1/4}(1+q^{1/2})\frac{(1-ab)(a-b)q^n}{2(1-abq^n)(1-abq^{n+1})}\;.
\end{align*}
Using this in \eqref{eq4S}, we see that by taking successively $n=3,4,\ldots$, the common denominator of $C_{n+1}$ is $(1-abq^{n+1/2})(1-abq^{n+1})^2(1-abq^{n+3/2})$, and therefore his numerator should have a similar form in order to fulfill the assumption.
Hence we find general solution of \eqref{eq4S} in the form
$$C_{n+1}=\frac{(1-v_1q^{n+1})(1-v_2q^{n+1})(1-v_3q^{n+1})(1-v_4q^{n+1})}{4(1-abq^{n+1/2})(1-abq^{n+1})^2(1-abq^{n+3/2})}\;,$$
for some complex numbers $v_i$, $i=1,2,3,4$. Putting this in \eqref{eq4S}, we find $v_1=1$, $v_2=a^2b^2$, $v_3=a^2$ and $v_4=b^2$. Therefore we obtain
\begin{align*}
C_{n+1}=\frac{(1-q^{n+1})(1-a^2q^{n+1})(1-b^2q^{n+1})(1-a^2b^2q^{n+1})}{4(1-abq^{n+1/2})(1-abq^{n+1})^2(1-abq^{n+3/2})}\;.
\end{align*}
Again this does not provide a solution to \eqref{lastcase-deg-pi-is-2}. Indeed, since we have $r_n=t_n+\alpha_n-\alpha_{n-1}=\widehat{b}q^{-n/2}(1-abq^n)$, we use initial conditions \eqref{b1-pi-2-case-2}--\eqref{b2-pi-2-case-2} to obtain a system similar to \eqref{solve-init1}--\eqref{solve-init3} and then proceed exactly as in the previous case to deduce that the new obtained system has no solutions for $\widehat{b}$ and $c$. Hence the result follows.
\end{proof}

\begin{remark}
It is important to notice that from \eqref{full-equation-chebyshev-case1} we can use Lemma \ref{lemma-for-alternative-relation} to obtain
\begin{align*}
(x^2-1)\mathcal{D}_qT_n(x)=\gamma_n\Big(T_{n+1}(x)-\frac{1}{4}T_{n-1}(x)\Big)\;. 
\end{align*}
This relation was also obtained in \cite{KCDMJP2021c}.
Similarly from \eqref{full-solution-Al-salam-Chihara}, we obtain
\begin{align*}
(\alpha^2-1)(x^2-1)\mathcal{D}_qQ_{n}&(x;s,sq^{1/2}|q)=(\alpha^2-1)\gamma_nQ_{n+1}(x;s,sq^{1/2}|q)\\
&+\big(c_{n+1}-\alpha c_n +(1-\alpha)\alpha_n B_n\big)Q_n(x;s,sq^{1/2}|q)\\
&+\big((B_n-\alpha B_{n-1})c_n +(1-\alpha^2)\gamma_n C_n \big)Q_{n-1}(x;s,sq^{1/2}|q)\\
&+(c_{n-1}C_n-\alpha c_nC_{n-1})Q_{n-2}(x;s,sq^{1/2}|q) \;,
\end{align*}
where $B_n$, $c_n$ and $C_n$ are given by \eqref{Bn-Cn-solution-Al-Salam-Chihara-firstcase-k1=0}. We also obtain from \eqref{full-equation-for-continuousdual-qhahn-firstcase}
\begin{align*}
(\alpha^2-1)(x^2-1)&\mathcal{D}_qH_{n}(x;s,-s,sq^{1/4}|q^{1/2})=(\alpha^2-1)\gamma_nH_{n+1}(x;s,-s,sq^{1/4}|q^{1/2})\\
&+\big(c_{n+1}-\alpha c_n +(1-\alpha)\alpha_n B_n\big)H_n(x;s,-s,sq^{1/4}|q^{1/2})\\
&+\big((B_n-\alpha B_{n-1})c_n +(1-\alpha^2)\gamma_n C_n \big)H_{n-1}(x;s,-s,sq^{1/4}|q^{1/2})\\
&+(c_{n-1}C_n-\alpha c_nC_{n-1})H_{n-2}(x;s,-s,sq^{1/4}|q^{1/2}) \;,
\end{align*}
where where $B_n$, $c_n$ and $C_n$ are given by \eqref{Bn-solution-continuous-dual-q-hahn}--\eqref{cn-solution-continuous-dual-qhahn}. This equation for the case $s=1$ therein is of type \eqref{0.2Dq-general} with $\deg \pi =2$, $r=1$ and $s=2$. We then recover the counterexample to Conjecture \ref{IsmailConj2478} presented in \cite{KCDM2022}. 
\end{remark}

\section*{Acknowledgments}
This work is partially supported by ERDF and Consejer\'ia de Econom\'ia, Conocimiento, Empresas y Universidad de la Junta de Andaluc\'ia (grant UAL18-FQM-B025-A) and by the Research Group FQM-0229 (belonging to Campus of International Excellence CEIMAR).


\begin{thebibliography}{10}

\bibitem{Al-Salam1995}
{W. Al-Salam},
{\it A characterization of the Rogers $q$-Hermite polynomials},
Internat. J. Math. $\&$ Math. Sci. \textbf{18} (1995) 641-648.



\bibitem{KCDMJP2021c}
{K. Castillo, D. Mbouna, and J. Petronilho}, {A characterization of continuous q-Jacobi, Chebyshev of the first kind and Al-Salam Chihara polynomials}, J. Math. Anal. Appl. {\bf 514} (2022) 126358.
\bibitem{KCDMJP2022}
{K. Castillo, D. Mbouna, and J. Petronilho},
{On the functional equation for classical orthogonal polynomials on lattices}, J. Math. Anal. Appl. {\bf 515} (2022) 126390.

\bibitem{KCDM2022} K. Castillo and D. Mbouna, A counterexample to a conjecture of M. Ismail, arXiv:2206.08375 [math.CA] (2022).





\bibitem{IsmailBook2005}
{M. E. H. Ismail},
{Classical and quantum orthogonal polynomials in one variable. With two chapters by W. Van Assche.
With a foreword by R. Askey.}, Encyclopedia of Mathematics and its Applications \textbf{98}.
Cambridge University Press, Cambridge, 2005.




\end{thebibliography}
\end{document}